\documentclass[11pt,a4paper]{article}

\usepackage[T1]{fontenc}
\usepackage[utf8]{inputenc}
\usepackage{lmodern}
\usepackage[margin=28mm]{geometry}
\usepackage{graphicx}
\usepackage{amsmath,amssymb,amsfonts,mathtools,amsthm}
\usepackage{color,array,xcolor}
\usepackage{booktabs}
\usepackage{units}
\usepackage{comment}
\usepackage{caption,subcaption}
\usepackage{microtype}
\usepackage{xparse}
\usepackage{tikz}
\usepackage{pgfplots}
\usepackage{enumitem}
\usepackage[linesnumbered,lined,ruled,commentsnumbered]{algorithm2e}

\definecolor{linkblue}{RGB}{0,70,140}

\usepackage[
    colorlinks=true,
    linkcolor=linkblue,
    citecolor=linkblue,
    urlcolor=linkblue
]{hyperref}

\newcommand{\R}{\mathbb R}
\newcommand{\C}{\mathbb C}
\newcommand{\N}{\mathbb N}
\DeclareMathOperator{\im}{im}
\newcommand{\lenQ}{Q}

\newtheorem{theorem}{Theorem}
\newtheorem{lemma}[theorem]{Lemma}
\newtheorem{proposition}[theorem]{Proposition}
\newtheorem{remark}[theorem]{Remark}
\newtheorem{definition}[theorem]{Definition}
\newtheorem{corollary}[theorem]{Corollary}

\begin{document}

\title{Data informativity for feedback rectification of switched linear systems}

\author{Philipp Schmitz$^{1}$, Hannes Gernandt$^{2,3}$,\\ Maria C. Honecker$^{4,5}$, and Karl Worthmann$^{1}$%
\thanks{$^{1}$Optimization-based Control Group, Technische Universit\"at Ilmenau, 98693 Ilmenau, Germany.
$^{2}$University of Wuppertal, Gau\ss stra\ss e 20, 42119 Wuppertal, Germany.
$^{3}$Fraunhofer IEG, Fraunhofer Research Institution for Energy Infrastructures and Geotechnologies IEG, 03046 Cottbus, Germany.
$^{4}$Chair of Intelligent Control Systems, RWTH Aachen University, Kopernikusstra\ss e 16, 52062 Aachen, Germany.
$^{5}$cigus GmbH, Magirus-Deutz-Stra\ss e 13, 89077 Ulm, Germany.}%
\thanks{Corresponding author: Philipp Schmitz (e-mail: \href{mailto:philipp.schmitz@tu-ilmenau.de}{philipp.schmitz@tu-ilmenau.de}).}%
\thanks{PS gratefully acknowledges funding by the Carl Zeiss Foundation (VERNEDCT - Project No. 2021-10-003). HG was funded by the Deutsche Forschungsgemeinschaft (DFG, German Research Foundation) -- Project-ID~531152215 -- CRC~1701. KW gratefully acknowledges support by the DFG -- Project-ID~535860958.}%
}

\date{}
\maketitle

\begin{abstract}
    For switched linear systems, we leverage the concept of data informativity for data-driven feedback rectification to design a stabilizing controller. 
    The established rectified eigenstructure means that all closed-loop subsystems share a common set of eigenvectors and enables us to compute a common quadratic Lyapunov function to ensure global uniform exponential stability of the closed-loop system. The proposed methodology is illustrated for two examples including a~model of an unmanned aerial vehicle.
\end{abstract}

\par

\section{Introduction}
Switched systems provide a natural modeling framework for dynamical processes whose evolution is governed by different modes of operation. Such switching phenomena occur in a broad range of applications, including power converters in electrical systems~\cite{trenn2012switched}, 
aircraft control systems~\cite{WangZS2016}, precision motion systems~\cite{HeertjesN2012}, electro-hydraulic systems~\cite{YuanDSB2017}, and coupled mechanical systems~\cite{Sajja2019} to name but a few. Accordingly, switched systems have been investigated extensively during the last decades, resulting in several monographs and a large body of stability and control-theoretic results, see, e.g., the textbooks~\cite{Liberzon.2003,zhao2017control} or the survey paper~\cite{lin2009stability}. 

A central difficulty in the analysis of switched systems is that, in general, stability depends not only
on the stability properties of the individual subsystems, but also on the switching signal. In particular, even switching between stable subsystems may lead to instability~\cite{Liberzon.2003} without additional dwell-time conditions \cite{Hespanha99}. 
This motivates the study of stability under arbitrary switching, that is, stability guarantees which are independent of the particular switching sequence. 
A standard tool for establishing such guarantees is the construction of common Lyapunov functions~\cite{mason2023universal}. Among these, common quadratic Lyapunov functions (CQLFs) are especially attractive, since they provide verifiable and computationally tractable certificates for stability under arbitrary switching~\cite{andersen2022common}.

With the development of stability criteria for switched systems, increasing attention has been devoted to feedback design. Existing approaches include methods based on linear matrix inequalities \cite{Sun05} and $H_\infty$-control tools \cite{Deaecto11} 
as well as model-reference adaptive control~\cite{YuanDSB2017} and switched model-predictive control~\cite{ZhangZB2016}. Other contributions address feedback designs that guarantee stability under broad classes of switching signals while also aiming at additional performance objectives; examples include composite nonlinear feedback designs for switched nonlinear~\cite{LinPB1998} and networked systems~\cite{li2018improved,zhao2022event}.

A particularly constructive line of work relies on eigenstructure assignment. In contrast to purely numerical design approaches, eigenstructure-based methods are closely related to classical linear control techniques and provide direct insight into the closed-loop dynamics. 

For linear time-invariant systems with multiple inputs, assigning the
closed-loop poles by state feedback generally leaves freedom in the associated
closed-loop eigenvectors. Moore's seminal result~\cite{Moore76} characterizes
the admissible eigenvectors for prescribed eigenvalues and thereby makes
explicit the design freedom beyond pole placement. Subsequent parametric
formulations have exploited this freedom, for example, to optimize the
robustness and gain of pole-placing feedbacks, including in the presence of
repeated poles~\cite{Schmid14}, and to connect eigenstructure assignment with
geometric-control constructions~\cite{padula2021eigenstructure}.

Related recent work provides explicit parameterizations of generalized
eigenvector chains, which are needed when the prescribed closed-loop
eigenstructure contains nontrivial Jordan chains
\cite{maynard5761099parameterizing}.

For switched linear systems under arbitrary switching, eigenstructure-based
feedback design has been used to impose common structural properties on the
closed-loop modes. The approach in~\cite{WulWirSho09} employs invariant-subspace
constructions to stabilize a class of single-input single-output switched systems, whereas
\cite{Hai16} studies the simultaneous triangularization of the closed-loop
matrices together with arbitrary eigenvalue assignment. The feedback
rectification framework of~\cite{HoneckerGWTR24} instead constructs
mode-dependent feedback matrices such that the closed-loop subsystems share a
common eigenvector basis, applying Moore's eigenstructure-assignment result to
the individual modes.

If the assigned eigenvalues are stable, the common eigenvector basis yields a
CQLF and hence stability under arbitrary switching. At the same time, the
remaining freedom in the eigenvector selection can be exploited for additional
performance objectives, as demonstrated for monotone tracking
in~\cite{HoneckerSWR24} and non-overshooting output shaping
in~\cite{WulffHSR26}.

The existing feedback rectification methods are model-based: they require exact knowledge of the subsystem matrices and involve the computation of basis vectors from intersections of eigenvector-assignment subspaces. While analytical descriptions of these intersections are available in certain cases, for example for two subsystems, their computation becomes increasingly involved as the number of modes grows. This motivates the question of whether feedback rectification can be carried out directly from measured data, without first relying on an explicit model-based representation of all subsystems.

Note that feedback rectification leads to simultaneous diagonalization of the closed-loop matrices, which is a sufficient but not necessary structural condition for stability under arbitrary switching and therefore introduces conservatism~\cite{Liberzon.2003}. More generally, even the existence of a CQLF is sufficient but not necessary for uniform exponential stability of switched systems \cite{ShortenEtAl2007}. Less restrictive alternatives include searching directly for common Lyapunov functions using LMI-based methods \cite{BoydEtAl1994} or employing multiple and nonquadratic Lyapunov functions \cite{ShortenEtAl2007}. The advantage of the present approach is its constructive character: the feedback enforces a joint eigenstructure that guarantees stability and can be incorporated naturally into the data-driven framework considered here.

Data-driven control has become an active research area, in particular following Willems' fundamental lemma~\cite{Willems2005}, which allows the behavior of controllable linear time-invariant systems to be represented from sufficiently rich input--state or input--output data, see also the review papers~\cite{markovsky2021behavioral,Faulwasser23}. 
This perspective has led to numerous data-driven stabilization and control methods for linear systems; see, for example,~\cite{van2020willems} for stabilization and~\cite{celi23} for pole assignment. More recently, data-driven methods have also been applied to switched linear systems in discrete time~\cite{eising2024data}, in continuous time \cite{bianchi2025data}, for output feedback control~\cite{hu2023data}, and also to nonlinear systems~\cite{Ju2026} using Koopman operator theory~\cite{mezic2026koopman}. These data-driven control approaches can be used for challenging applications, including security-related problems such as the design of controllers defending against false data injection~\cite{Liu25}.

These approaches typically rely on data-informativity conditions to obtain stabilizing controllers from data, possibly without knowing the switching signal during the data collection phase. The setting considered in the present paper is different and complementary: our goal is not only to stabilize the switched system, but also to construct mode-dependent feedbacks that rectify the closed-loop eigenstructure. This provides an explicit CQLF and allows the designer to prescribe desired closed-loop eigenvalue locations.

In this paper, we propose a data-driven method for the rectification and stabilization of switched linear systems. We work in discrete time and assume that, during an offline phase, sufficiently informative input--state data are collected for each subsystem. Based on these data, we derive algebraic conditions that characterize the feasibility of a common closed-loop eigenvector basis and provide explicit formulas for the corresponding rectifying feedback matrices. The resulting closed-loop systems share the same eigenvectors, while their eigenvalues may be chosen mode-dependently inside the unit disk. Consequently, a CQLF can be constructed directly from the data, yielding global uniform exponential stability under arbitrary switching, provided that the active mode is available for the implementation of the corresponding feedback.

The main contributions of this paper are summarized as follows.
\begin{itemize}
    \item We formulate feedback rectification for discrete-time switched linear systems with possibly repeated eigenvalues and relate the existence of stabilizing rectifying feedback matrices to a common closed-loop eigenstructure. 
    \item Using input--state data and data informativity, we derive a data-based representation of the eigenvector assignment subspaces that appear in Moore's eigenstructure assignment theorem.
    \item We provide a data-driven characterization of the parametric subspace conditions required for feedback rectification. 
    \item We obtain an explicit data-driven formula for the mode-dependent feedback matrices. 
    \item We show how the rectified eigenstructure yields a CQLF directly from data.
    \item For prescribed closed-loop eigenvalues, we propose an algorithm that verifies the data informativity and computes the rectifying feedbacks.
    \item For eigenvalues constrained only to a stability region, we formulate a nonlinear program that determines eigenvalues and eigenvectors simultaneously while promoting a well-conditioned eigenvector matrix.
\end{itemize}

The paper is organized as follows. Section~\ref{sec:prelim} recalls preliminaries on switched systems, CQLFs, Moore's eigenstructure assignment theorem, and model-based feedback rectification. Section~\ref{sec:main_data} contains the main data-driven results, including the data-based rectifiability condition, the feedback formula, and the resulting algorithms for constructing stabilizing rectifying feedbacks and a CQLF. Section~\ref{sec:example} illustrates the proposed methods by numerical examples. Finally, Section~\ref{sec:conclusion} concludes the paper and outlines directions for future work.\\

\noindent\textbf{Notation}.
We denote by $\mathbb{N}$, $\mathbb{R}$, and $\mathbb{C}$ the sets of nonnegative integers, real numbers, and complex numbers, respectively. For a vector $x\in\mathbb{C}^n$, $x^\top$ denotes its transpose and $x^*$ its conjugate transpose. The Euclidean norm is denoted by $\Vert x\Vert$. For a matrix $A\in\mathbb{C}^{n\times m}$, we denote by $A^*$ its conjugate transpose and by $A^\dagger$ its Moore--Penrose pseudoinverse. The image and kernel of $A$ are denoted by $\im A$ and $\ker A$, respectively. Furthermore, $\operatorname{rank} A$ denotes the rank of $A$, i.e., the dimension of $\im A$. For a square matrix $A\in\mathbb{C}^{n\times n}$, $\rho(A)$ denotes its spectral radius, i.e., the maximum modulus of its eigenvalues, while $\|A\|$ and $\|A\|_\mathrm{F}$ denote the spectral and Frobenius norms, respectively. Given $A=[\begin{smallmatrix}
    a_1&\dots & a_n 
\end{smallmatrix}]\in\R^{m\times n}$ we define $\operatorname{vec}(A):=[\begin{smallmatrix}
    a_1^\top & \dots & a_n^\top
\end{smallmatrix}]^\top\in\R^{nm}$. The identity matrix in $\R^{n\times n}$ is denoted by $I_n$.
The finite set of switching modes is denoted by $
    \mathcal{Q}:=\{1,\ldots,\lenQ\},$ where $\lenQ\in\mathbb{N}$ is the total number of modes. For each $q\in\mathcal{Q}$, the matrix $A_q$ denotes the system matrix associated with mode $q$. A switching signal is a map
$
    \sigma:\mathbb{N}\to\mathcal{Q},
$
where $\sigma(k)$ specifies the active mode at time step $k$. 
Finally, $\pi:\{1,\ldots,n\}\to\{1,\ldots,n\}$ denotes an involution, that is, a self-inverse mapping satisfying 
$\pi(\pi(i))=i$, $i\in\{1,\ldots,n\}.$

\section{Preliminaries}
\label{sec:prelim}
In this section, we recall stability, stabilization and existence conditions for CQLFs in Subsection~\ref{ssec:cqlfs}. Subsection~\ref{ssec:feedbackrectification} presents the feedback rectification together with a~sufficient condition based on  \cite{HoneckerGWTR24}.
\subsection{Exponential stability, stabilization, and CQLFs}
\label{ssec:cqlfs}
We consider a switched linear system that is composed of the constituent discrete-time linear time-invariant (LTI) systems given by
\begin{equation}
    x_{k+1} = A_{q} x_k + B_{q} u_k,\quad k\geq 0,
\end{equation}
with $A_q\in\R^{n\times n}$, $B_q\in\R^{n\times m}$ and $q\in\mathcal Q=\{1,2,\dots, \lenQ\}$, where $k\in\N$. These subsystems constitute the discrete-time switched system
\begin{equation}
    \label{eq:switched_sys}
    x_{k+1} = A_{\sigma(k)} x_k + B_{\sigma(k)} u_k,\quad k\geq 0,
\end{equation}
with switching signal $\sigma: \N \to \mathcal Q$. For the stability analysis we consider the homogeneous system 
\begin{align}
    \label{eq:homo}
 x_{k+1} = A_{\sigma(k)}x_{k},\quad  x_0=x^0, \quad k\geq 0,
\end{align}
and call the origin 
\emph{globally uniformly exponentially stable} if there exist $C\geq 0$ and $\gamma\in(0,1)$ such that  {$\|x_k\|\leq C\gamma^k\|x^0\|$} holds for all $k\geq 0$ and all $x^0\in\mathbb{R}^n$.
A \emph{common quadratic Lyapunov function} (CQLF) is a function $\mathcal{V}(x)=\tfrac{1}{2}x^\top Px$ for some symmetric positive definite $P\in\mathbb{R}^{n\times n}$ which is decreasing along  {all trajectories of~\eqref{eq:homo}, for every switching signal}, i.e., there exists $\alpha>0$ such that $\mathcal{V}(x_{k+1})-\mathcal{V}(x_{k})\leq -\alpha \|x_k\|^2$ for all $k\geq 0$.

The aim of this paper is to design, based on input--state data, a feedback law of the form
\begin{equation}
    \label{eq:feedback_law}
    u_k = F_{\sigma(k)} x_k
\end{equation} with $F_{\sigma(k)}\in \R^{m\times n}$ which exponentially stabilizes the closed-loop switched system
\begin{equation}
    x_{k+1} = (A_{\sigma(k)} + B_{\sigma(k)}F_{\sigma(k)}) x_k,\quad k\geq 0,
\end{equation} 
under  {arbitrary switching signals} $\sigma$, by constructing a CQLF for the closed-loop system.

In the following proposition, 
we recall a result on the exponential stability of the homogeneous system \eqref{eq:homo} using CQLFs from \cite{SunShorten2005}.
\begin{proposition}
\label{prop:lyapunov}
Consider the homogeneous switched system~\eqref{eq:homo} with $\rho(A_q)<1$ for all $q\in\mathcal Q$. If there exists a nonsingular $T\in\C^{n\times n}$ such that $T^{-1}A_q T$ is upper triangular for all $q\in\mathcal Q$, then there exists a common quadratic Lyapunov function. Consequently, the origin of~\eqref{eq:homo} is  {globally uniformly exponentially stable}. If, in addition, $T^{-1}A_q T$ is diagonal for all $q\in\mathcal Q$, then the Lyapunov function is given by 
    \[
    \mathcal{V}(x)=\tfrac{1}{2}x^\top T^{-*}T^{-1}x\quad \text{for all $x\in\mathbb{R}^n$}.
    \]
\end{proposition}

\subsection{Model-based feedback rectification}
\label{ssec:feedbackrectification}
The design of the switched feedback law~\eqref{eq:feedback_law} relies on the seminal result of Moore \cite{Moore76}, which characterizes the  {eigenstructure} assignment for a single LTI system
\begin{equation}
    \label{eq:LTI_sys}
    x_{k+1} = A x_k + B u_k
\end{equation}
with respect to the closed-loop matrix $A + BF$. For $\lambda\in\C$ consider the matrices $N(\lambda)$ and $M(\lambda)$ satisfying
\begin{equation}
    \label{eq:NlambMlamb}
    \operatorname{ker}\begin{bmatrix}
        \lambda I - A & B
    \end{bmatrix} = \operatorname{im}\begin{bmatrix}
        N(\lambda)\\ M(\lambda)
    \end{bmatrix}.
\end{equation}

The following result from \cite{KleinMoore77} extends \cite{Moore76} to repeated eigenvalues.
\begin{theorem}
\label{thm:moore}
Let $(A,B)\in \mathbb R^{n\times n}\times \mathbb R^{n\times m}$ and let $\Lambda=\{\lambda_1,\dots,\lambda_n\}$ be self-conjugate. Then there exists $F\in \mathbb R^{m\times n}$ such that $(A+BF)v_i = \lambda_i v_i$, $v_i\neq 0$, for all $i\in\{1,\dots, n\}$ if and only if the following three conditions are satisfied:
\begin{itemize}
    \item[\rm (i)] $\{v_1,\dots, v_n\}\subset \mathbb C^{n}$ is linearly independent;
    \item[\rm (ii)] there exists an involution $\pi$ such that
\[
    \lambda_{\pi(i)}=\overline{\lambda_i},
    \qquad
    v_{\pi(i)}=\overline{v_i},
    \qquad i=1,\ldots,n;
\]
    \item[\rm (iii)] $v_i\in \operatorname{im} N(\lambda_i)$ for all $i=1,\ldots,n$,
\end{itemize}
where $\bar v$ and $\bar \lambda$ denote the complex conjugate of $v$ and $\lambda$, respectively. 
\end{theorem}

If the conditions of Theorem~\ref{thm:moore} are satisfied, the feedback matrix $F$ can be computed explicitly. Specifically, for eigenvectors $v_i = N(\lambda_i) w_i$ with suitable $w_i$, the feedback matrix $F$ satisfies
\begin{equation}
    \label{eq:F}
    F\begin{bmatrix}
        v_1 & \dots & v_n
    \end{bmatrix} = -\begin{bmatrix}
        M(\lambda_1)w_1 & \dots & M(\lambda_n)w_n
        \end{bmatrix}.
\end{equation}
Since $\begin{bmatrix}
        v_1 & \dots & v_n
    \end{bmatrix}$ is invertible, we can rewrite~\eqref{eq:F} as
\begin{align}
\label{eq:F_explicit_with_inverse}
F=-\begin{bmatrix}
        M(\lambda_1)w_1 &\!\! \dots &\!\! M(\lambda_n)w_n
        \end{bmatrix}\begin{bmatrix}
        v_1 &\!\! \dots &\!\! v_n
    \end{bmatrix}^{-1}.
\end{align}
\begin{remark}
\label{rem:real_feedback}
Note that the matrix $F$ in \eqref{eq:F_explicit_with_inverse} indeed has real entries, i.e., $F\in\mathbb{R}^{m\times n}$, as stated in Theorem~\ref{thm:moore}. This can be seen by separating the columns into real and imaginary parts. Here, without loss of generality, we assume that the eigenvectors are sorted in such a way that the first $2n_1$ eigenvectors correspond to the non-real eigenvalues and that they appear pairwise with a vector and its complex conjugate, i.e.,  {$v_{2i-1}=\overline{v_{2i}}$} for $i=1,\ldots,n_1$, and that the remaining $n-2n_1$ eigenvectors correspond to the real eigenvalues, i.e., $v_{i}=\overline{v_{i}}$ for $i=2n_1+1,\ldots,n$. With this assumption and using the abbreviation $v=v^R+\mathrm{i}v^I$ for the decomposition into real and imaginary parts, i.e., $v^R, v^I\in\mathbb{R}^n$, and writing $\hat{w}_i:= M(\lambda_i)w_i$ for $i=1,\ldots,n$, we obtain 
\begin{align*}
&~~~~\begin{bmatrix}
        M(\lambda_1)w_1 &\!\! \dots &\!\! M(\lambda_n)w_n
        \end{bmatrix}\begin{bmatrix}
        v_1 &\!\! \dots &\!\! v_n
    \end{bmatrix}^{-1}\\
    & = \left[\begin{smallmatrix} \hat{w}_1
^R+ \mathrm{i}\hat w_1^I &\hat{w}_1
^R- \mathrm{i}\hat w_1^I  & \ldots &  {\hat{w}_{2n_1-1}
^R- \mathrm{i}\hat w_{2n_1-1}^I} &\hat{w}_{2n_{1}+1}&\ldots & \hat w_n
        \end{smallmatrix}\right]\\ & ~~~\cdot\begin{bmatrix}
        v_1^R+{\rm i} v_1^I &  v_1^R-{\rm i} v_1^I  &\!\! \dots & v_{2n_1+1}& \ldots&\!\! v_n
    \end{bmatrix}^{-1} \\
    &= \begin{bmatrix} \hat w_1^R& \hat w_1^I & \ldots &  {\hat w_{2n_1-1}^R}&  {\hat w_{2n_1-1}^I} & \hat w_{2n_1+1} & \ldots &\!\!  \hat w_n
        \end{bmatrix} \\ &~~~~\left[\begin{smallmatrix}
\mathrm{diag}\left(\begin{bmatrix}
            1 & 1 \\ \mathrm{i} & -\mathrm{i}
        \end{bmatrix}\right)   &0\\0& I_{n-2n_1}      \end{smallmatrix} \right]\left[\begin{smallmatrix}
\mathrm{diag}\left(\begin{bmatrix}
            1 & 1 \\ \mathrm{i} & -\mathrm{i}
        \end{bmatrix}\right)   &0\\0& I_{n-2n_1}      \end{smallmatrix}\right]^{-1}\\
        &~~~~\begin{bmatrix}
        v_1^R & v_1^I &\!\! \dots &\!\!  {v_{2n_1-1}^R} &  {v_{2n_1-1}^I}& v_{2n_1+1} & \ldots & v_n
    \end{bmatrix}^{-1}
    \\&=\begin{bmatrix} \hat w_1^R& \hat w_1^I \, \ldots \, {\hat w_{2n_1-1}^R}&  {\hat w_{2n_1-1}^I} & \hat w_{2n_1+1} \,\ldots \, \hat w_n
        \end{bmatrix}\\
        &~~~~\begin{bmatrix} v_1^R& v_1^I \, \ldots \,  {v_{2n_1-1}^R}&  {v_{2n_1-1}^I} & v_{2n_1+1}\,  \ldots \,  v_n
        \end{bmatrix}^{-1} \in\mathbb{R}^{m\times n}.
\end{align*}
Applying suitable permutation matrices, the same calculation shows that $F$ is also real when the eigenvectors are not sorted pairwise.
\end{remark}

We recall the definition of feedback rectification from \cite{HoneckerGWTR24}, slightly extended to more than two subsystems. To this end, we introduce some notation.
For a given set of complex numbers $\{\lambda_{q,i}\}_{q,i=1}^{Q,n}$ consider for  $q\in\mathcal Q$
\begin{equation}
\label{eq:prescribed_modal_targets} \Lambda_q:=\operatorname{diag}
\bigl(\lambda_{q,1},\ldots,\lambda_{q,n}\bigr)\in\C^{n\times n}, 
\end{equation}
 {i.e., $\lambda_{q,i}$ is the desired $i$-th closed-loop eigenvalue of subsystem $q$. Collecting, for each $i$, these eigenvalues across all subsystems yields} the indexed pole tuples
\[
 \boldsymbol\lambda_i
 :=(\lambda_{1,i},\ldots,\lambda_{\lenQ,i})\in\C^{\lenQ},
 \qquad i=1,\ldots,n.
\]
 {We assume that these tuples} are compatible with complex conjugation, i.e., there exists an
involution $\pi$ of $\{1,\ldots,n\}$ such that
$\boldsymbol\lambda_{\pi(i)}=\overline{\boldsymbol\lambda_i}$. A matrix $V=[\,v_1\ \cdots\ v_n\,]$ is called
 \emph{conjugation-compatible} if
$v_{\pi(i)}=\overline{v_i}$ for all $i$. The same terminology is
used for matrices whose columns are indexed in the same way.  {We emphasize that $\{\Lambda_q\}_{q\in\mathcal Q}$, $\{\lambda_{q,i}\}_{q,i=1}^{Q,n}$ and $(\boldsymbol\lambda_i)_{i=1}^n$ encode the same
eigenvalue data.}

\begin{definition}
\label{def:fbrect}
Let $A_q\in\R^{n\times n}$ and $B_q\in\R^{n\times m}$, $q\in\mathcal{Q}$, and let $\{\boldsymbol{\lambda}_i\}_{i=1}^n\subset\C^{\lenQ}$.
Then the family $\{(A_q,B_q)\}_{q\in \mathcal{Q}}$ is called \emph{feedback rectifiable with $\{\Lambda_q\}_{q=1}^Q$} if there exist $F_q\in\R^{m\times n}$, 
and a linearly independent set $\{v_1,\ldots,v_n\}$ in $\C^n$ such that the following holds
\begin{align}
\label{def:ferec}
(A_q+B_qF_q)v_i=\lambda_{q,i}v_i,
\end{align}
for all $i=1,\ldots,n$ and all $q=1,\ldots,\lenQ$. 
\end{definition}

A sufficient condition for feedback rectifiability is given by the following result, which is a slight generalization of \cite{HoneckerGWTR24} to switched systems with more than two subsystems and follows directly from Theorem~\ref{thm:moore}. As in the single-system case, we consider for $\lambda\in\C$ and $q\in\mathcal Q$ matrices $N_q(\lambda)$, $M_q(\lambda)$ satisfying
\begin{equation}
\label{eq:image_with_q}
    \ker \begin{bmatrix}
        \lambda I-A_q & B_q
    \end{bmatrix} = \im \begin{bmatrix}
        N_q(\lambda)\\
        M_q(\lambda)
    \end{bmatrix}.
\end{equation}

\begin{proposition}
\label{prop:rect}
Let $A_q\in\R^{n\times n}$ and $B_q\in\R^{n\times m}$, $q\in\mathcal{Q}$.  
Then the family $\{(A_q,B_q)\}_{q\in\mathcal Q}$ is feedback
rectifiable with  {$\{\Lambda_q\}_{q=1}^Q$} given by \eqref{eq:prescribed_modal_targets} if there exist indexed eigenvalue tuples
\[
    \boldsymbol{\lambda}_i=\bigl(
        \lambda_{1,i},\ldots,\lambda_{\lenQ,i}
    \bigr)
    \in \mathbb{C}^Q,
    \qquad i=1,\ldots,n,
\]
an involution $\pi$,
and linearly independent vectors
$v_1,\ldots,v_n\in\mathbb C^n$ such that, for every
$i=1,\ldots,n$, we have  $v_i\in
        \bigcap_{q=1}^{\lenQ}
        \operatorname{im}N_q\bigl(\lambda_{q,i}\bigr)$
     and $V=[\,v_1\ \cdots\ v_n\,]$ is conjugation-compatible.
\end{proposition}

\begin{remark}
In the case of two subsystems, i.e., $\lenQ=|\mathcal{Q}|=2$, one can use reduced row echelon forms of matrices of rational functions in two variables to compute the subspace intersection in Proposition~\ref{prop:rect} explicitly \cite{HoneckerGWTR24}. In principle, this can be extended to more than two subsystems, although this might lead to complicated expressions of rational functions in the subspace representation.  
\end{remark}

\section{Data-driven feedback rectification}
\label{sec:main_data}
In this section, we present our main results on data-driven feedback rectification. First, we recall in Subsection~\ref{ssec:informativity} the data-informativity approach of \cite{van2020data} in the context of feedback rectification. In Subsection~\ref{ssec:numerical_rectifiability}, we discuss how feedback rectifiability can be verified numerically for prescribed eigenvalues. In Subsection~\ref{ssec:nonlinear_program}, we present a nonlinear optimization approach which determines the eigenvalues and promotes a well-conditioned eigenvector matrix.

\subsection{Data informativity and feedback rectification}
\label{ssec:informativity}

For every mode $q\in\mathcal{Q}$ let $T_q$ denote the number of
recorded transitions and define
\begin{equation}
\begin{split}
\label{eq:data_matq}
 U_q&=[\,u_{q,0}\ \cdots\ u_{q,T_q-1}\,]\in  \R^{m\times T_q},\\
 X_q&=[\,x_{q,0}\ \cdots\ x_{q,T_q-1}\,]\in\mathbb R^{n\times T_q},\\
 X_q^+&=[\,x_{q,1}\ \cdots\ x_{q,T_q}\,]\in\mathbb R^{n\times T_q}, 
\end{split}
\end{equation}
which satisfy
\begin{equation}
\label{eq:exact_data_equation}
X_q^+=A_qX_q+B_qU_q.
\end{equation}
Following the data-informativity framework of \cite{van2020data}, we introduce for the data $\mathcal D_q:=(U_q,X_q,X_q^+)$ the \emph{consistency set}
\begin{equation*}
 \Sigma_q(\mathcal D_q)
 :=\left\{(A,B)\,:\, X_q^+ = A X_q + BU_q\right\}.
\end{equation*}
We assume $\Sigma_q(\mathcal D_q)\neq\emptyset$ for all
$q\in\mathcal Q$. The following lemma, which follows from basic linear algebra \cite{BenIsraelGreville2003}, characterizes this assumption in terms of data.

\begin{lemma}
\label{lem:data_consistency}
Consider data matrices $U_q, X_q, X_q^+$ in \eqref{eq:data_matq} and $\mathcal D_q:=(U_q,X_q,X_q^+)$. Then the following statements are equivalent:
\begin{enumerate}[label=(\roman*)]
\item
$\Sigma_q(\mathcal D_q)\neq\emptyset$;
\item
$\ker \begin{bmatrix}
X_q\\U_q
\end{bmatrix}\subseteq\ker X_q^+$;
\item
$X_q^+\left(I-\begin{bmatrix}
X_q\\U_q
\end{bmatrix}^\dagger \!\!\begin{bmatrix}
X_q\\U_q
\end{bmatrix}\right)=0$;
\item
$\operatorname{rank}
\begin{bmatrix}
X_q^\top \, U_q^\top \, (X_q^+)^\top
\end{bmatrix}^\top 
=
\operatorname{rank}\begin{bmatrix}
X_q\\U_q
\end{bmatrix}. $
\end{enumerate}
If these conditions hold, all systems compatible with the data are parameterized by
\begin{align}
\label{eq:parametrization_sigma}
\begin{bmatrix}A \;\, B\end{bmatrix}
=
X_q^+\begin{bmatrix}
X_q\\U_q
\end{bmatrix}^\dagger
+
Z\left(I-\begin{bmatrix}
X_q\\U_q
\end{bmatrix}\!\!\begin{bmatrix}
X_q\\U_q
\end{bmatrix}^\dagger\right),
\end{align}
where $Z\in\mathbb R^{n\times(n+m)}$ is arbitrary.
\end{lemma}

\begin{definition}
\label{def:informative_rectification}
Consider data matrices $U_q, X_q, X_q^+$ in \eqref{eq:data_matq}.  
The data $\{\mathcal D_q\}_{q\in\mathcal Q}$ are called 
\emph{informative for feedback rectification with $\{\Lambda_q\}_{q=1}^Q$ given by  \eqref{eq:prescribed_modal_targets} } if there exist a nonsingular,
conjugation-compatible matrix $V\in\mathbb C^{n\times n}$ and
matrices $F_q\in\mathbb R^{m\times n}$ such that
\begin{equation}
\label{eq:informative_rectification}
 (A+BF_q)V=V\Lambda_q
\end{equation}
for every $(A,B)\in\Sigma_q(\mathcal D_q)$ and every
$q\in\mathcal Q$.
\end{definition}

\begin{theorem}
\label{thm:informativity_rectification}
 Consider data matrices $U_q, X_q, X_q^+$ in \eqref{eq:data_matq} such that  $\Sigma_q(\mathcal D_q)\neq\emptyset$ for all
$q\in\mathcal Q$. Then the data are informative for feedback rectification with $\{\Lambda_q\}_{q=1}^Q$ given by  \eqref{eq:prescribed_modal_targets} if and only if there exist a
nonsingular conjugation-compatible matrix
$V\in\mathbb C^{n\times n}$ and conjugation-compatible matrices
$G_q\in\mathbb C^{T_q\times n}$ such that
\begin{equation}
\label{eq:informativity_G}
 X_qG_q=V,
 \qquad
 X_q^+G_q=V\Lambda_q,
 \qquad q\in\mathcal Q.
\end{equation}
In this case, real rectifying feedbacks are given by
\begin{equation}
\label{eq:F_data}
 F_q=U_qG_qV^{-1},
 \qquad q\in\mathcal Q,
\end{equation}
and satisfies $A+BF_q=V\Lambda_qV^{-1}$ for every $(A,B)\in\Sigma_q(\mathcal D_q)$.
\end{theorem}

\begin{proof}
Assume first that the data are informative and let $F_q$ and $V$
satisfy~\eqref{eq:informative_rectification}. Fix
$(A_q^0,B_q^0)\in\Sigma_q(\mathcal D_q)$. 
For every matrix $\Delta$ satisfying $\Delta \big[\begin{smallmatrix}X_q\\U_q\end{smallmatrix}\big]=0$, the matrix
$[\,A_q^0\ B_q^0\,]+\Delta$ is again compatible with the data.
Hence,~\eqref{eq:informative_rectification} implies
$\Delta \big[\begin{smallmatrix}I_n\\F_q\end{smallmatrix}\big]=0$ whenever $\Delta \big[\begin{smallmatrix}X_q\\U_q\end{smallmatrix}\big]=0$. Equivalently,
\[
 \ker \begin{bmatrix}X_q\\U_q\end{bmatrix}^\top\subseteq\ker \begin{bmatrix}I_n\\F_q\end{bmatrix}^\top,
\]
and therefore
$\operatorname{im}\big[\begin{smallmatrix}I_n\\F_q\end{smallmatrix}\big]\subseteq\operatorname{im}\left[\begin{smallmatrix}X_q\\U_q\end{smallmatrix}\right]$. Thus, there
exists a real matrix $H_q\in\mathbb R^{T_q\times n}$ such that
\[
 \begin{bmatrix}X_q\\U_q\end{bmatrix}H_q=\begin{bmatrix}I_n\\F_q\end{bmatrix}.
\]
Set $G_q:=H_qV$. Then
$X_qG_q=V$ and, using the data equation,
\[
 X_q^+G_q
 =[\,A_q^0\ B_q^0\,]\begin{bmatrix}X_q\\U_q\end{bmatrix}H_qV
 =(A_q^0+B_q^0F_q)V
 =V\Lambda_q.
\]
Moreover, $G_q$ is conjugation-compatible because $H_q$ is real.

Conversely, suppose~\eqref{eq:informativity_G} holds and define
$F_q$ by~\eqref{eq:F_data}. Conjugation compatibility gives
$\overline V=VP_\pi$ and $\overline{G_q}=G_qP_\pi$, where
$P_\pi$ is the permutation matrix associated with $\pi$. Hence,
\[
 \overline{G_qV^{-1}}
 =G_qP_\pi(VP_\pi)^{-1}=G_qV^{-1},
\]
so $F_q$ is real. For every
$(A,B)\in\Sigma_q(\mathcal D_q)$,
\[
 \begin{aligned}
 (A+BF_q)V
 =AX_qG_q+BU_qG_q
 =X_q^+G_q
 =V\Lambda_q.
 \end{aligned}
\]
\end{proof}

For $\lambda\in\mathbb C$ and $\boldsymbol\lambda=(\lambda_1,\ldots,\lambda_\lenQ)\in\C^{\lenQ}$  define
\begin{equation*}
 \mathcal E_q^{\mathcal D}(\lambda)
 :=X_q\ker(X_q^+-\lambda X_q),\quad 
 \mathcal E^{\mathcal D}(\boldsymbol\lambda)
 :=\bigcap_{q\in\mathcal Q}
 \mathcal E_q^{\mathcal D}(\lambda_q).
\end{equation*}

\begin{corollary}\label{cor:fixed_poles_informativity}
Consider data matrices $U_q, X_q, X_q^+$ in \eqref{eq:data_matq} such that  $\Sigma_q(\mathcal D_q)\neq\emptyset$ for all
$q\in\mathcal Q$.  Then the data are informative for  {feedback rectification with $\{\Lambda_q\}_{q=1}^Q$} if and only if
there exist linearly independent, conjugation-compatible vectors
\begin{equation}
\label{eq:assignable_common_vectors}
 v_i\in\mathcal E^{\mathcal D}(\boldsymbol\lambda_i),
 \qquad i=1,\ldots,n.
\end{equation}
If $g_{q,i}$ satisfy
\begin{equation}
\label{eq:column_lifts}
 X_qg_{q,i}=v_i,
 \qquad
 (X_q^+-\lambda_{q,i}X_q)g_{q,i}=0,
\end{equation}
then, with
$V=[\,v_1\ \cdots\ v_n\,]$ and
$G_q=[\,g_{q,1}\ \cdots\ g_{q,n}\,]$, the feedback is given by
\eqref{eq:F_data}.
\end{corollary}

\begin{proof}
The two matrix identities in~\eqref{eq:informativity_G} are
equivalent to \eqref{eq:column_lifts} for every $i$ and $q$.
Thus
\[
 v_i\in X_q\ker(X_q^+-\lambda_{q,i}X_q)
 =\mathcal E_q^{\mathcal D}(\lambda_{q,i})
\]
for every $q$, and hence
$v_i\in\mathcal E^{\mathcal D}(\boldsymbol\lambda_i)$.
 The invertibility of $V$ is equivalent to linear independence of the
vectors $v_1,\ldots,v_n$. Conversely, let~\eqref{eq:assignable_common_vectors} hold. By the
definition of $\mathcal E^{\mathcal D}$, for each $i$ and $q$ there
exists $g_{q,i}$ satisfying~\eqref{eq:column_lifts}. The vectors $v_i$ can be
chosen conjugation-compatible because all data matrices are real.
Stacking the vectors $g_{q,i}$ yields matrices $G_q$ satisfying
\eqref{eq:informativity_G}. Applying Theorem~\ref{thm:informativity_rectification}
then gives the result.
\end{proof}

A necessary condition for informativity follows directly from~\eqref{eq:informativity_G} in Theorem~\ref{thm:informativity_rectification}. Since $X_q G_q = V$ with $V$ nonsingular, the matrix $G_qV^{-1}$ is a right inverse of $X_q$, and hence,
\[
\operatorname{rank} X_q = n\quad \text{for all }q\in\mathcal Q.
\]
In the following, we consider the special case in which the consistency set $\Sigma_q(\mathcal{D}_q)$ is a singleton. This holds under an additional rank condition, which is satisfied if $(A_q,B_q)$ is controllable and the input sequence  {$u_q=(u_{q,0},\dots, u_{q,T_q-1})$} is
persistently exciting of order $n + 1$, see, e.g., Willems et al.'s fundamental lemma \cite[Theorem~1]{Willems2005} and \cite[Theorem~1]{van2020willems}.
\begin{corollary}
\label{cor:full_rank_rectification}
Consider data matrices $U_q, X_q, X_q^+$ in \eqref{eq:data_matq}.  {Then the rank condition}
\begin{equation}
\label{eq:rankq}
\operatorname{rank}\begin{bmatrix}X_q\\U_q\end{bmatrix}
 =n+m,
 \qquad q\in\mathcal Q,
\end{equation}
 {holds if and only if} $\Sigma_q(\mathcal D_q)=\{(A_q,B_q)\}$ with
\begin{equation}
\label{eq:identified_model}
 [\,A_q\ B_q\,]=X_q^+\begin{bmatrix}
     X_q\\ U_q
 \end{bmatrix}^\dagger.
\end{equation}
 {Moreover, if \eqref{eq:rankq} holds and} $N_q(\lambda)\in\mathbb{C}^{n\times k}$, $M_q(\lambda)\in\mathbb{C}^{m\times k}$ are such that \eqref{eq:image_with_q} holds,
then
\begin{equation}
\label{eq:data_kernel}
 \mathcal E_q^{\mathcal D}(\lambda)
 =\operatorname{im}N_q(\lambda).
\end{equation}
Consequently, the system \eqref{eq:switched_sys} is feedback rectifiable with $\{\Lambda_q\}_{q=1}^Q$ 
if
and only if there exist linearly independent,
conjugation-compatible vectors
\begin{equation}
\label{eq:model_based_intersection_recovered}
 v_i\in\bigcap_{q\in\mathcal Q}
 \operatorname{im}N_q(\lambda_{q,i}),
 \qquad i=1,\ldots,n.
\end{equation}
The corresponding feedback matrices are given by~\eqref{eq:F_data}.
\end{corollary}

\begin{proof}
Fix $q\in\mathcal Q$. Since $\big[\begin{smallmatrix}
     X_q\\ U_q
 \end{smallmatrix}\big]$ has full row
rank, it follows that  $\big[\begin{smallmatrix}
     X_q\\ U_q
 \end{smallmatrix}\big]\big[\begin{smallmatrix}
     X_q\\ U_q
 \end{smallmatrix}\big]^\dagger=I_{n+m}$ and every
$(A_q,B_q)\in\Sigma_q(\mathcal D_q)$ satisfies
\[
 [\,A_q\ B_q\,]
 =X_q^+\begin{bmatrix}
     X_q\\ U_q
 \end{bmatrix}^\dagger.
\]
Thus, the data determine a unique pair $(A_q,B_q)$ and
\eqref{eq:identified_model} holds. Conversely, if $\Sigma_q(\mathcal D_q)$ contains only $(A_q,B_q)$, then the parametrization \eqref{eq:parametrization_sigma} implies that $I-\big[\begin{smallmatrix}
X_q\\U_q
\end{smallmatrix}\big] \big[\begin{smallmatrix}
X_q\\U_q
\end{smallmatrix}\big]^\dagger=0$. This implies $\im\big[\begin{smallmatrix}
X_q\\U_q
\end{smallmatrix}\big]=\C^{n+m}$, since $\big[\begin{smallmatrix}
X_q\\U_q
\end{smallmatrix}\big] \big[\begin{smallmatrix}
X_q\\U_q
\end{smallmatrix}\big]^\dagger$ is the orthogonal projector onto $\im\big[\begin{smallmatrix}
X_q\\U_q
\end{smallmatrix}\big]$. Thus, \eqref{eq:rankq} holds.

Let $\lambda\in\mathbb C$. From \eqref{eq:exact_data_equation} 
we find
\[
 X_q^+-\lambda X_q
 =-[\,\lambda I_n-A_q\quad B_q\,]
 \begin{bmatrix}X_q\\-U_q\end{bmatrix}.
\]
The matrix $[\,X_q^\top\ -U_q^\top\,]^\top$ has full row
rank. Therefore
\[
 \ker[\,\lambda I_n-A_q\quad B_q\,]
 =\begin{bmatrix}X_q\\-U_q\end{bmatrix}
 \ker(X_q^+-\lambda X_q).
\]
Taking the first $n$ components gives
\[
 \operatorname{im}N_q(\lambda)
 =X_q\ker(X_q^+-\lambda X_q)
 =\mathcal E_q^{\mathcal D}(\lambda),
\]
which proves~\eqref{eq:data_kernel}. Hence,
\[
 \mathcal E^{\mathcal D}(\boldsymbol\lambda_i)
 =\bigcap_{q\in\mathcal Q}
   \operatorname{im}N_q(\lambda_{q,i}),
\]
and the assertion follows from
Corollary~\ref{cor:fixed_poles_informativity}.
\end{proof}

The spaces $\mathcal E_q^{\mathcal D}(\lambda)$ are defined via kernels of $n\times T_q$ matrices. Since Subsection~\ref{ssec:numerical_rectifiability} requires repeated evaluations for different candidate eigenvalue tuples, we eliminate the dependence on the data length $T_q$. To this end, we consider for each $q\in\mathcal{Q}$ the left annihilator 
\begin{equation}
\label{eq:left_annihilator}
 L_q\!=\![\,L_{q,0}\; L_{q,1}\,],
 \;
 L_q\!\begin{bmatrix}X_q\\X_q^+\end{bmatrix}\!=\!0,
 \;
 \ker L_q\!=\!\operatorname{im}\begin{bmatrix}X_q\\X_q^+\end{bmatrix}\!,
\end{equation}
where $L_q$ has full row rank and both blocks have $n$ columns. For
$\boldsymbol\lambda=(\lambda_1,\ldots,\lambda_\lenQ)$ set
\begin{equation}
\label{eq:curlyX}
 \Phi(\boldsymbol\lambda)
 :=\begin{bmatrix}
 L_{1,0}+\lambda_1 L_{1,1}\\
 \vdots\\
 L_{\lenQ,0}+\lambda_\lenQ L_{\lenQ,1}
 \end{bmatrix}
\end{equation}
whose dimension 
is independent of the number of collected data samples. Hence, repeated rectifiability tests for different prescribed eigenvalue tuples involve only matrices with $n$ columns.
\begin{lemma}
\label{lem:reduced_eigenspace}
For every $\boldsymbol\lambda\in\mathbb C^{\lenQ}$, we have
$ \mathcal E^{\mathcal D}(\boldsymbol\lambda)
 =\ker\Phi(\boldsymbol\lambda)$.
\end{lemma}
\begin{proof}
For a fixed $q$, a vector $v$ belongs to
$\mathcal E_q^{\mathcal D}(\lambda_q)$ if and only if there is a
vector $g_q$ such that
\[
 \begin{bmatrix}X_q\\X_q^+\end{bmatrix}g_q
 =\begin{bmatrix}v\\\lambda_q v\end{bmatrix}.
\]
By~\eqref{eq:left_annihilator}, this is equivalent to
\[
 (L_{q,0}+\lambda_q L_{q,1})v=0.
\]
Imposing these equations for every mode yields the claimed result.
\end{proof}

\subsection{Verification for prescribed real eigenvalues}
\label{ssec:numerical_rectifiability}

We present a method for the numerical verification of feedback rectifiability, given a prescribed set of candidate eigenvalues. Here, we focus on real eigenvalues. The complex case is discussed at the end of the subsection in Remark~\ref{rem:complex}.

\begin{proposition}
\label{prop:rado_rectification}
Let
$\boldsymbol\lambda_i\in\mathbb R^{\lenQ}$,
$i=1,\ldots,n$, be prescribed real pole tuples. For every $i$, let
 {$S_i\in\mathbb R^{n\times n_i}$} have linearly independent columns and
satisfy
\begin{equation}
\label{eq:Si_rado}
\operatorname{im}S_i=\ker\Phi(\boldsymbol\lambda_i).
\end{equation}
Then the following statements are equivalent:
\begin{itemize}
 \item[(i)] the data are informative for feedback rectification
 with the prescribed pole tuples;
 \item[(ii)] for every nonempty
 $I\subseteq\{1,\ldots,n\}$,
 \begin{equation}
 \label{eq:rado_rank_test}
  \operatorname{rank}[\,S_i\,]_{i\in I}\ge |I|;
 \end{equation}
 \item[(iii)] there exist indices
  {$j_i\in\{1,\ldots,n_i\}$} such that, if $s_{i,j_i}$ denotes the
 $j_i$th column of $S_i$, then
 \begin{equation}
 \label{eq:transversal_V}
  V=[\,s_{1,j_1}\ \cdots\ s_{n,j_n}\,]
 \end{equation}
 is nonsingular.
\end{itemize}
If these conditions hold, then for every $q\in\mathcal Q$ the linear
system
\begin{equation}
\label{eq:G_recovery}
 \begin{bmatrix}X_q\\X_q^+\end{bmatrix}G_q
 =\begin{bmatrix}V\\V\Lambda_q\end{bmatrix}
\end{equation}
is feasible and $F_q=U_qG_qV^{-1}$ is a rectifying feedback.
\end{proposition}

\begin{proof}
By Lemma~\ref{lem:reduced_eigenspace},
$\operatorname{im}S_i$ is precisely the space of common eigenvectors
compatible with the prescribed tuple $\boldsymbol\lambda_i$.
Corollary~\ref{cor:fixed_poles_informativity} therefore reduces
informativity to the existence of linearly independent vectors
$v_i\in\operatorname{im}S_i$.

Rado's theorem~\cite{Rado1942} states that such an independent
transversal exists if and only if
\[
 \dim\left(\sum_{i\in I}\operatorname{im}S_i\right)\ge |I|
 \qquad\text{for every }I\subseteq\{1,\ldots,n\}.
\]
Since
$\dim\left(\sum_{i\in I}\operatorname{im}S_i\right)
 =\operatorname{rank}[\,S_i\,]_{i\in I}$, this proves the equivalence of (i) and (ii).
Applying Rado's theorem to the finite sets of basis columns of the
matrices $S_i$ yields the equivalence with (iii).

Finally, $v_i\in\ker\Phi(\boldsymbol\lambda_i)$ and
Lemma~\ref{lem:reduced_eigenspace} imply, for every $q$, the existence
of $g_{q,i}$ with
\[
 X_qg_{q,i}=v_i,
 \qquad
 X_q^+g_{q,i}=\lambda_{q,i}v_i.
\]
Stacking these vectors yields~\eqref{eq:G_recovery}. The feedback
formula follows from Theorem~\ref{thm:informativity_rectification}.
\end{proof}
For $i=1,\ldots,n$, let $S_i=[s_{i,1}\ \cdots\ s_{i,n_i}]$ be a basis matrix with $\operatorname{im}S_i=\ker\Phi(\boldsymbol\lambda_i)$, where $n_i=\dim\ker\Phi(\boldsymbol\lambda_i)$. We select linearly independent columns by searching a tree $\mathcal G=(\mathcal V, \mathcal E)$ whose vertices are given by
\begin{equation*}
    \mathcal V = \{()\}\,\!\cup\,\! \bigcup_{k=1}^n\left\{(j_1,\dots, j_k)\,:\, \begin{aligned}
         &j_i\in\{1,\dots, n_i\},\\
 &\operatorname{rank}
 [s_{1,j_1}\ \cdots\ s_{k,j_k}]=k \end{aligned}\right\}\!.
\end{equation*}
The root of $\mathcal G$ is the empty tuple $()$. For a tuple  {$\nu=(j_1,\dots, j_k)$}, we write  {$(\nu,j):= (j_1,\dots, j_k, j)$} for the tuple obtained by appending $j$ to  {$\nu$}, so that $((),j) = (j)$. Every vertex is adjacent to each of its one-entry extensions,
\begin{equation*}
    \mathcal E = \Big\{\big( {\nu},( {\nu},j)\big) \,:\,  {\nu}\in\mathcal V, ( {\nu},j)\in\mathcal V\Big\}.
\end{equation*}
There is a one-to-one correspondence between length-$n$ vertices $(j_1,\dots, j_{n})\in\mathcal V$ and the nonsingular matrices $V=[\,s_{1,j_1}\,\dots\, s_{n,j_n}\,]$ whose $i$th column is a column of $S_i$, $i=1,\dots,n$. By Proposition~\ref{prop:rado_rectification} the existence of such a vertex is therefore equivalent to informativity for feedback rectification. If $(j_1,\dots, j_{n})\in\mathcal V$, then every truncation $(j_1,\dots, j_k)$, $k=1,\dots, n-1$, also belongs to $\mathcal V$, since subsets of linearly independent vectors are linearly independent. Hence, every length-$n$ vertex is connected to the root $()$ by the path $(), (j_1),(j_1,j_2),\dots, (j_1,\dots, j_n)$ and can therefore be found by depth-first search \cite{IntroductionAlgorithms2022}. The vertex set $\mathcal V$ need not be constructed in advance. The search builds $\mathcal V$ dynamically by testing the rank condition for each one-entry extension $(j_1,\dots,j_k,j)$, $1\le j\le n_{k+1}$, of a visited vertex $(j_1,\dots,j_k)\in\mathcal V$ and keeping only those that pass. Algorithm~\ref{alg:1} combines the informativity verification and column selection.

\begin{algorithm}[t]
\caption{Rectification for prescribed real pole tuples by
depth-first search}
\label{alg:1}
\KwData{
Consistent data $\{\mathcal D_q\}_{q\in\mathcal Q}$
and real pole tuples $\boldsymbol\lambda_i$,
$i=1,\ldots,n$.
}
\KwResult{
A nonsingular matrix $V$ and rectifying feedback matrices
$\{F_q\}_{q\in\mathcal Q}$, or failure of informativity
for the prescribed tuples.
}
\For{$q\in\mathcal Q$}{
 compute a left annihilator $L_q$ with full row rank from
\eqref{eq:left_annihilator}\;
}
\For{$i=1,\ldots,n$}{
 compute a basis matrix
 $S_i=[s_{i,1}\ \cdots\ s_{i,n_i}]$
 of $\ker\Phi(\boldsymbol\lambda_i)$\;
 \If{$n_i=0$}{
  \Return{failure of informativity}\;
 }
}
perform depth-first search on $\mathcal G$ starting in the root to find a length-$n$ vertex\;
\If{a length-$n$ vertex  {$\nu=(j_1,\dots, j_n)$} is found}{
$V\gets \begin{bmatrix}s_{1,j_1}&\dots&s_{n,j_n}\,\end{bmatrix}$\;
\For{$q\in\mathcal Q$}{
 solve \eqref{eq:G_recovery} for $G_q$\;
 compute $F_q$ by solving $F_qV=U_qG_q$\;
}
\Return{$V$, $\{F_q\}_{q\in\mathcal Q}$, and $P=V^{-\top} V^{-1}$}\;
}
 \Return{failure of informativity}\;
\end{algorithm}

\begin{remark}    
\label{rem:complex}
Complex-conjugate pole tuples can be treated by replacing every
conjugate pair $a_{q,i}\pm\mathrm i b_{q,i}$ by the real block
\begin{equation*}
 \Theta_{q,i}:=
 \begin{bmatrix}
  a_{q,i}&b_{q,i}\\
  -b_{q,i}&a_{q,i}
 \end{bmatrix}.
\end{equation*}
If  {$v_i=v_i^R+\mathrm i v_i^I$}, then the two complex eigenvector equations
are equivalent to
\[
 (A_q+B_qF_q)[\,v_i^R\ v_i^I\,]
 =[\,v_i^R\ v_i^I\,]\Theta_{q,i}.
\]
Thus, the feedback recovery can be carried out over the reals with a
block-diagonal modal matrix. The two columns belonging to a conjugate
pair have to be selected jointly.
\end{remark}
\begin{remark}
\label{rem:complexity_alg1}
Regarding the complexity, note that the number of rank tests is at most
$\sum_{k=1}^{n}\prod_{i=1}^{k}n_i$, which is acceptable for the state dimensions $n=6$ and $n=7$ considered in Section~\ref{sec:example}. For larger state dimensions, the construction of an invertible matrix $V$ in Algorithm~\ref{alg:1} can alternatively be formulated as the intersection of a linear matroid and a partition matroid, for which polynomial-time algorithms are available; see, e.g., \cite{BercziKiralYamaguchiYokoi2023}.
\end{remark}

\begin{remark}
Although we assume that the active mode is known at every time step during closed-loop operation, the proposed scheme can be combined with the data-driven mode detection approach in~\cite{eising2024data}. If the CQLF value ceases to decrease, this may indicate an unobserved mode switch. In this case, mode detection can be triggered to identify the active subsystem and select the corresponding feedback gain. This extension requires additional data-informativity assumptions and a detection procedure that is sufficiently fast relative to the system dynamics. Under these conditions, the approach can also be applied when the switching signal is not known a priori.
\end{remark}

\subsection{Rectification with free real eigenvalues}
\label{ssec:nonlinear_program}
Computing kernel elements of $\Phi(\boldsymbol\lambda_i)$ is subject to the additional constraint that the assembled matrix $V$ be invertible. This constraint is difficult to enforce with standard linear-algebraic techniques and, in practice, often leads to ill-conditioned $V$.  {Moreover, Subsection~\ref{ssec:numerical_rectifiability} requires the eigenvalues to be prescribed.} To address these issues, we propose the following nonlinear optimization problem for a prescribed stability bound $\alpha\in(0,1)$:
\begin{subequations}
\label{eq:opt}
\begin{align}
&\underset{
\{\boldsymbol\lambda_i\}_{i=1}^n,
V,W
}{\operatorname{minimize}}
\quad
\frac{1}{2}\left(\|V\|_\mathrm{F}^2 + \|W\|_\mathrm{F}^2\right)
     \label{eq:opta}\quad\text{subject to}\\
&-\alpha
\le
\lambda_{q,i}
\le
\alpha,
\quad
(q,i)\in \mathcal Q\times \{1,\ldots,n\}
    \label{eq:optb}\\
    \label{eq:optc}
&c_1(\{\boldsymbol{\lambda}_i\}_{i=1}^n, V) := \begin{bmatrix}
    \Phi(\boldsymbol{\lambda}_1)\\
    &\ddots\\
    &&\Phi(\boldsymbol{\lambda}_n)
\end{bmatrix}\operatorname{vec}(V) = 0\\
\label{eq:optd}
&c_2(V,W) := VW - I_n = 0 .
\end{align}
\end{subequations}

\noindent Constraint~\eqref{eq:optd} ensures invertibility of every feasible $V$ with $V^{-1}=W$, while the objective function promotes well-conditioning of $V$, since
\begin{equation*}
\begin{split}
    \|V\|_\mathrm{F}^2 + \|W\|_\mathrm{F}^2 &= \|V\|_\mathrm{F}^2 + \|V^{-1}\|_\mathrm{F}^2 = \sum_{k=1}^n (\sigma_k^2 + \sigma_k^{-2})\\
    &\geq \sigma_1^2 + \sigma_n^{-2} \geq 2 \sigma_1 \sigma_n^{-1},
\end{split}
\end{equation*} where $\sigma_1\geq \dots \geq \sigma_n>0$ are the singular values of $V$. The user-defined parameter $\alpha\in(0,1)$ in \eqref{eq:opt} ensures that the constructed feedback stabilizes the closed-loop system. Although the objective function is  {convex}, the  {bilinear} constraints~\eqref{eq:optc},~\eqref{eq:optd} render the optimization problem~\eqref{eq:opt} nonconvex.

To solve the constrained optimization problem~\eqref{eq:opt}, we apply the augmented Lagrangian method \cite{nocedal2006numerical}. Define
\[
c(\{\boldsymbol\lambda_i\}_{i=1}^n,
V,W):= \begin{bmatrix}
    c_1(\{\boldsymbol\lambda_i\}_{i=1}^n,V)\\
\operatorname{vec}\big(c_2(V,W)\big)
\end{bmatrix}\in\R^\eta
\]
and let $\mu\in \R^{\eta}$, where $\eta = n^2 + n\sum_{q\in\mathcal Q} r_q$ with $r_q = \operatorname{rank} L_q = 2n - \operatorname{rank} \big[\begin{smallmatrix}
    X_q\\X_q^+
\end{smallmatrix}\big]$. Given  {$\beta>0$}, the augmented Lagrangian reads
\begin{equation*}
\begin{split}
    &\mathcal L(\{\boldsymbol\lambda_i\}_{i=1}^n,
V,W, \mu;  {\beta}):=\\
    &\frac{1}{2}\Big(\lVert V\rVert_\mathrm{F}^2 + \lVert W\rVert_\mathrm{F}^2\Big) + \mu^\top c(\{\boldsymbol\lambda_i\}_{i=1}^n,
V,W)\\
&+  {\frac{\beta}{2}} \lVert c(\{\boldsymbol\lambda_i\}_{i=1}^n,
V,W)\rVert^2.
\end{split}
\end{equation*}
Then the minimization problem
\begin{equation}
\label{eq:aug_lagr}
    \operatorname*{minimize}_{\{\boldsymbol\lambda_i\}_{i=1}^n,
V,W}\ \mathcal L(\{\boldsymbol\lambda_i\}_{i=1}^n,
V,W, \mu_k;  {\beta_k})
\end{equation}%
subject to the bound constraint~\eqref{eq:optb} is solved successively via the  {bound-constrained} limited-memory BFGS  {(L-BFGS-B)} algorithm, see, e.g., \cite{nocedal2006numerical}. The penalty factors  {$\beta_k>0$} form a nondecreasing sequence chosen adaptively with  {$\beta_k\rightarrow\infty$} and the Lagrange multipliers $\mu_k$ are updated appropriately between iterations.  {Collecting the decision variables in $z := \big(\{\boldsymbol\lambda_i\}_{i=1}^n, V, W\big)$, the $k$th subproblem is solved starting from the approximate minimizer $z_{k-1}$ of the previous one, where the initial guess $z_0$ has to be provided by the user. The resulting sequence of iterates approximates a (local) minimizer of the nonlinear program~\eqref{eq:opt}.}

\begin{algorithm}[t]
\caption{Data-driven feedback rectification via optimization}
\label{alg:2}

\KwData{
Consistent data $\{\mathcal D_q\}_{q\in\mathcal Q}$  {and an initial guess $z_0 = \big(\{\boldsymbol\lambda_i^0\}_{i=1}^n, V_0, W_0\big)$ for the decision variables of~\eqref{eq:opt}}.
}

\KwResult{
A nonsingular matrix $V$, eigenvalues $(\boldsymbol{\lambda}_i)_{i=1}^n$ and rectifying feedback matrices
$\{F_q\}_{q\in\mathcal Q}$.
}

\For{$k=1,2,\dots, k_{\max}$}{
    find an approximate minimizer $z_k$ of \eqref{eq:aug_lagr} starting at $z_{k-1}$\;
    \If{converged}{
        $\big(\{\boldsymbol{\lambda}_i\}_{i=1}^n, V,W\big)\gets z_k$\;
        \For{$q\in\mathcal Q$}{
        solve \eqref{eq:G_recovery} for $G_q$\;
        compute $F_q$ by solving $F_qV=U_qG_q$\;
        }
        \Return{$V$, $(\boldsymbol{\lambda}_i)_{i=1}^n$, $\{F_q\}_{q\in\mathcal Q}$, and $P=V^{-\top}V^{-1}$}\;
    }
    update $\mu_{k+1}$  {and $\beta_{k+1}$}\;
}

\Return{not converged, no solution found}\;
\end{algorithm}

\begin{remark}
Since~\eqref{eq:opt} is nonconvex, Algorithm~\ref{alg:2} may converge to a local minimizer or fail to converge at all. The first scenario is not critical, since we are primarily interested in a feasible point, i.e., a point satisfying the constraints \eqref{eq:optb}--\eqref{eq:optd}, which already yields a rectifying feedback. In particular, Algorithm~\ref{alg:2} may be terminated before its convergence criterion is met, as soon as the constraint violation falls below a prescribed tolerance. However, such a feasible point may come at the expense of a large condition number of $V$.  A failure of Algorithm~\ref{alg:2}, on the other hand, does not necessarily mean that the data are not informative for feedback rectification with real eigenvalues within the interval $[-\alpha,\alpha]$. In this case, the optimization may be restarted with a different  {initial guess $z_0$} or an increased stability bound $\alpha$.
\end{remark}
\begin{remark}
    Note that both the number of decision variables and the number of constraints in problem~\eqref{eq:opt} grow linearly in the number of switching modes $\lenQ$ and quadratically in the state dimension $n$. In particular, the problem size does not depend on the sample sizes $T_q$ of the data $\mathcal D_q$, $q\in\mathcal Q$. Due to nonconvexity, the number of iterations for each augmented Lagrangian subproblem and therefore the runtime of Algorithm~\ref{alg:2} cannot be bounded a priori.
\end{remark}
\begin{remark} 
Complex-conjugate eigenvalues can be included in~\eqref{eq:opt} via the real parameterization of Remark~\ref{rem:complex}. To this end, the number $n_\mathrm{c}$ of conjugate pairs and the involution~$\pi$ (cf.\ Proposition~\ref{prop:rect}) are fixed in advance, where $\pi$ is the same for all switching modes since the eigenvectors are shared. Instead of \eqref{eq:optb}, a disk constraint $a_{q,i}^2 + b_{q,i}^2\leq \alpha^2$ is implemented via slack variables $a_{q,i}^2 + b_{q,i}^2 = s_{q,i}^2$ and additional box constraints $s_{q,i}\in [0,\alpha]$. Since the problem is invariant under simultaneous permutations of the columns of $V$, the rows of $W$, and the eigenvalues, only the number $n_\mathrm{c}$ matters, and at most $\lfloor n/2\rfloor + 1$ modified versions of problem~\eqref{eq:opt} have to be solved.
\end{remark}

\section{Numerical examples}
\label{sec:example}
We illustrate the proposed algorithms by means of two examples.
\subsection{Randomly chosen system matrices}
We consider a switched system with $\lenQ=4$ switching modes, input dimension $m=2$ and state dimension $n=6$. The subsystems are randomly generated from a construction that guarantees feedback rectifiability. All numerical values below are rounded to two  {decimal places}. The system matrices are given by 
\begin{align*}
    A_{1} &= \left[\begin{smallmatrix}
    -0.20 & -0.29 & 0.16 & -0.12 & 0.47 & 0.76 \\
    -0.88 & -0.48 & 0.23 & -0.22 & 0.22 & 0.77 \\
    -0.03 & -0.31 & 0.26 & -0.13 & -0.32 & 0.15 \\
    0.93 & 0.36 & -1.05 & 0.34 & -1.18 & -0.23 \\
    -0.75 & -0.05 & 0.72 & 0.05 & -0.97 & -0.27 \\
    -0.16 & 0.20 & 0.52 & 0.18 & 0.49 & -0.33
\end{smallmatrix}\right]\\
    A_{2} &= \left[\begin{smallmatrix}
    -0.96 & -0.73 & 0.22 & 0.64 & 0.76 & 0.07 \\
    0.59 & 0.12 & -0.38 & -0.75 & -0.65 & 0.11 \\
    -0.04 & -0.27 & 0.31 & 0.42 & 0.23 & 0.41 \\
    -1.34 & 0.01 & -0.22 & 0.44 & 1.55 & -2.21 \\
    0.86 & 0.20 & -0.37 & 0.48 & 0.53 & 0.32 \\
    0.27 & 0.06 & -0.07 & 0.79 & 0.35 & 0.81
\end{smallmatrix}\right]\\
    A_{3} &= \left[\begin{smallmatrix}
    -0.40 & 0.12 & -0.48 & -0.20 & -0.29 & -0.10 \\
    0.21 & -1.14 & 0.38 & 0.16 & -0.10 & -1.12 \\
    -0.02 & -0.26 & -0.59 & -0.19 & -0.01 & -0.84 \\
    0.31 & 0.93 & 0.35 & 0.71 & -0.33 & 2.80 \\
    -0.45 & -0.55 & -0.32 & -0.78 & -0.53 & -1.28 \\
    -0.04 & -0.21 & 0.46 & 0.41 & 0.26 & -0.12
\end{smallmatrix}\right]\\
 A_{4} &= \left[\begin{smallmatrix}
    0.43 & 0.11 & 0.99 & -0.47 & 0.90 & -1.62 \\
    0.16 & 0.05 & 1.28 & 0.09 & 0.69 & -0.98 \\
    -0.18 & -0.06 & -0.32 & 0.15 & -0.29 & 0.32 \\
    -0.06 & 0.38 & 0.07 & -1.45 & 0.64 & -1.78 \\
    1.09 & 0.43 & 0.74 & 0.87 & 0.51 & 0.57 \\
    0.07 & -0.11 & 0.16 & 0.57 & -0.09 & 0.52
\end{smallmatrix}\right]
\end{align*}
\begin{equation*}
\begin{aligned}
    B_{1} &= \left[\begin{smallmatrix}
    0.44 & -0.22 \\
    0.32 & -0.46 \\
    -0.25 & -0.27 \\
    -0.84 & 0.34 \\
    -0.33 & 0.01 \\
    0.34 & 0.31
\end{smallmatrix}\right]&
    B_{2} &= \left[\begin{smallmatrix}
    1.21 & -0.69 \\
    -0.17 & 0.54 \\
    0.30 & -0.26 \\
    -0.09 & -0.96 \\
    -0.17 & -0.19 \\
    0.16 & -0.39
\end{smallmatrix}\right]\\
    B_{3} &= \left[\begin{smallmatrix}
    -0.40 & 0.57 \\
    0.75 & 0.16 \\
    0.16 & 0.19 \\
    -0.60 & -0.10 \\
    -0.11 & 0.27 \\
    0.70 & -0.74
\end{smallmatrix}\right]&
    B_{4} &= \left[\begin{smallmatrix}
    -1.12 & -0.88 \\
    -1.12 & -0.56 \\
    0.31 & 0.37 \\
    0.19 & -0.99 \\
    -1.62 & 0.04 \\
    -0.33 & 0.32
\end{smallmatrix}\right].
\end{aligned}
\end{equation*}

For the data-driven rectification, we use data $\mathcal D_q=(U_q,X_q,X_q^+)$ satisfying~\eqref{eq:rankq}. Applying Algorithm~\ref{alg:2}, we find a switched state feedback with matrices
\begin{align*}
    F_{1} &= \left[\begin{smallmatrix}
    0.47 & -0.07 & -0.30 & -0.06 & -1.18 & -1.01 \\
    -0.93 & -0.70 & 0.28 & -0.63 & -0.26 & 0.91
\end{smallmatrix}\right] \\
    F_{2} &= \left[\begin{smallmatrix}
    0.33 & 0.71 & 0.11 & 0.25 & 0.04 & 0 \\
    -0.51 & 0.22 & -0.07 & 1.47 & 1.10 & 0.27
\end{smallmatrix}\right] \\
    F_{3} &= \left[\begin{smallmatrix}
    -0.09 & 1.04 & -0.52 & -0.58 & 0.14 & 0.59 \\
    0.01 & 0.93 & 0.60 & 0.59 & 0.39 & 1.63
\end{smallmatrix}\right] \\
    F_{4} &= \left[\begin{smallmatrix}
    0.49 & 0.10 & 0.53 & 0.16 & 0.37 & -0.49 \\
    -0.13 & 0.14 & 0.57 & -0.28 & 0.36 & -0.19
\end{smallmatrix}\right]
\end{align*}
The assigned common eigenvectors form the columns of
\[
V = \left[\begin{smallmatrix}
    0.76 & 0.43 & -0.47 & 0.40 & -0.03 & 0.11 \\
    -0.46 & 0.68 & 0.11 & -0.28 & 0.05 & 0.71 \\
    -0.18 & 0.24 & 0.15 & 0.70 & -0.11 & 0.11 \\
    0.41 & 0.85 & -1.87 & -0.76 & 0.16 & -1.37 \\
    -1.03 & -0.71 & 0.70 & 0.20 & 1.34 & -0.29 \\
    -0.52 & -0.70 & 0.60 & 0.36 & 0.18 & 0.54
\end{smallmatrix}\right],
\]
where $V$ has a condition number of $\sigma_1 \sigma_n^{-1}\approx 9.3$, and the corresponding Lyapunov matrix reads
\[
P = \left[\begin{smallmatrix}
    2.33 & 0.38 & -1.27 & 0.02 & 0.51 & 0.88 \\
    0.38 & 0.91 & -0.30 & 0.20 & 0.07 & 0.47 \\
    -1.27 & -0.30 & 2.68 & 0.32 & -0.24 & -0.30 \\
    0.02 & 0.20 & 0.32 & 1.18 & -0.51 & 2.65 \\
    0.51 & 0.07 & -0.24 & -0.51 & 0.80 & -1.48 \\
    0.88 & 0.47 & -0.30 & 2.65 & -1.48 & 7.53
\end{smallmatrix}\right].
\]
The eigenvalues of the closed-loop subsystems are listed in Table~\ref{tab:ev_rand}.
\begin{table}
\caption{Eigenvalues of the closed-loop matrices $A_q+B_q F_q$ of the random switched system.}
\label{tab:ev_rand}
\centering
\begin{tabular}{rrrrrrr}
\toprule
$q$ & $\lambda_{1}$ & $\lambda_{2}$ & $\lambda_{3}$ & $\lambda_{4}$ & $\lambda_{5}$ & $\lambda_{6}$ \\
\midrule
$1$ & $-0.62$ & $-0.31$ & $-0.29$ & $0.08$ & $0.23$ & $0.40$ \\
$2$ & $-0.57$ & $-0.32$ & $-0.07$ & $0.30$ & $0.39$ & $0.56$ \\
$3$ & $-0.63$ & $-0.63$ & $-0.56$ & $-0.18$ & $0.03$ & $0.47$ \\
$4$ & $-0.68$ & $-0.53$ & $-0.06$ & $0.12$ & $0.19$ & $0.26$ \\
\bottomrule
\end{tabular}
\end{table}
Fig.~\ref{fig:rand} depicts the Euclidean norm along a closed-loop state trajectory together with the values of the Lyapunov function under a random switching signal whose active mode is known at each time step. When the eigenvalues in Table~\ref{tab:ev_rand} are prescribed, Algorithm~\ref{alg:1} confirms that the data are informative for feedback rectification and returns a different eigenvector matrix. For perturbed eigenvalue candidates, in contrast, Algorithm~\ref{alg:1} reports failure of informativity.
\begin{figure}
    \centering
    \includegraphics[width=1.0\linewidth]{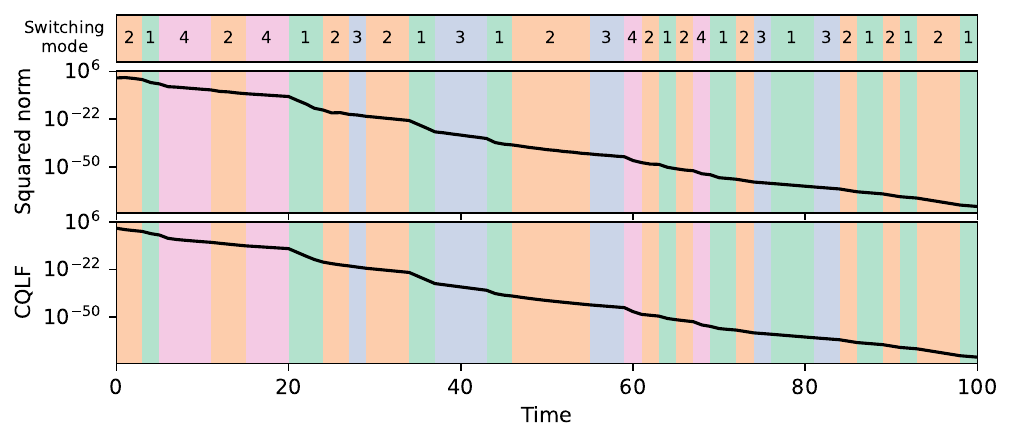}
    \caption{Random switched system. Euclidean norm and CQLF along a closed-loop state trajectory under a random switching signal.}
    \label{fig:rand}
\end{figure}

\subsection{Unmanned Aerial Vehicle (UAV)}
In \cite{honecker2025development}, an unmanned aerial vehicle (UAV) is modeled with $\lenQ = 2$ switching modes, input dimension $m=5$ and state dimension $n=7$. Mode one is the nominal mode, whereas mode two represents faulty operation, mainly characterized by a loss of yaw controllability. This example is chosen as in \cite{honecker2025development} it is demonstrated to be difficult to solve with \cite{HoneckerGWTR24} and restrictions on the eigenvalue pairs apply. All numerical values below are rounded to three  {decimal places}. The system matrices of the individual subsystems are given by
\begin{align*}
    A_{1} &= \left[\begin{smallmatrix}
    0.998 & -0.003 & -0.012 & -0.490 & 0 & 0 & 0 \\
    -0.005 & 0.004 & 0 & 0.002 & 0 & 0 & 0 \\
    -0.038 & -0.241 & 0.966 & 0.009 & 0 & 0 & 0 \\
    -0.001 & -0.010 & 0.049 & 1.000 & 0 & 0 & 0 \\
    0 & 0 & 0 & 0 & 0.957 & 0 & 0 \\
    0 & 0 & 0 & 0 & -0.137 & 0.450 & 0.226 \\
    0 & 0 & 0 & 0 & 0.190 & 0.017 & 0.605
\end{smallmatrix}\right] \\
    A_{2} &= \left[\begin{smallmatrix}
    0.998 & -0.005 & -0.012 & -0.490 & 0 & 0 & 0 \\
    -0.005 & 0.004 & 0 & 0.002 & 0 & 0 & 0 \\
    -0.043 & 0 & 0.966 & 0.011 & 0 & 0 & 0 \\
    -0.001 & 0 & 0.049 & 1.000 & 0 & 0 & 0 \\
    0 & 0 & 0 & 0 & 0.957 & 0 & 0 \\
    0 & 0 & 0 & 0 & -0.137 & 0.450 & 0.226 \\
    0 & 0 & 0 & 0 & 0.190 & 0.017 & 0.605
\end{smallmatrix}\right] \\
   B_{1} &= \left[\begin{smallmatrix}
    0.005 & 0 & 0 & 0.002 & -0.001 \\
    0.104 & 0 & 0 & 0 & 0.009 \\
    -1.871 & 0 & 0 & 0.047 & 0.293 \\
    -0.047 & 0 & 0 & 0.001 & 0.007 \\
    0 & 0 & -0.008 & 0 & 0 \\
    0 & 4.136 & 5.258 & 0 & 0 \\
    0 & 0.110 & -0.802 & 0 & 0
\end{smallmatrix}\right] \\
    B_{2} &= \left[\begin{smallmatrix}
    0.005 & 0 & 0 & 0.003 & -0.001 \\
    0.104 & 0 & 0 & 0 & 0 \\
    -1.757 & 0 & 0 & 0 & 0.303 \\
    -0.044 & 0 & 0 & 0 & 0.008 \\
    0.054 & 0 & 0 & 0 & 0 \\
    -0.005 & 4.136 & 0 & 0 & 0 \\
    0.006 & 0.110 & 0 & 0 & 0
\end{smallmatrix}\right].
\end{align*}

For this example, we use data $\mathcal D_q=(U_q, X_q, X_q^+)$ with $\operatorname{rank} X_q=n$, which is necessary for informativity, but  {$\operatorname{rank} [\,X_q^\top\ U_q^\top]^\top=n + m -1$} for $q=1,2$, deliberately violating the rank condition~\eqref{eq:rankq}. Algorithm~\ref{alg:2} yields the feedback matrices
\begin{align*}
    F_{1} &= \left[\begin{smallmatrix}
    -1.279 & -4.854 & 0.352 & 4.280 & -0.941 & -0.020 & 0.113 \\
    2.321 & 0.109 & -0.052 & -5.689 & -120.481 & -0.044 & -1.086 \\
    -1.673 & -0.079 & 0.037 & 4.143 & 97.673 & -0.184 & 0.799 \\
    -10.477 & 14.934 & -1.243 & 8.227 & 897.430 & -1.060 & 8.663 \\
    12.091 & -31.788 & -3.838 & -43.192 & -117.165 & 0.057 & -0.490
    \end{smallmatrix}\right] \\
    F_{2} &= \left[\begin{smallmatrix}
    -0.001 & 0 & 0 & 0.001 & -6.674 & 0.002 & 0.001 \\
    0.169 & 0.014 & -0.040 & -0.571 & 0.339 & 0.038 & -0.053 \\
    6.508 & 12.128 & -3.829 & -32.574 & -33.847 & 0.144 & -0.640 \\
    4.352 & 13.736 & -1.501 & -28.123 & 287.776 & -0.051 & 1.559 \\
    17.952 & 0.319 & -6.077 & -66.854 & -7.076 & 0.028 & 0.180
\end{smallmatrix}\right]
\end{align*}
with eigenvector matrix
\[
V = \left[\begin{smallmatrix}
    0.017 & -1.398 & 0.031 & -0.037 & 0.514 & 0.066 & 0.155 \\
    0 & 0.007 & 0.010 & 1.229 & -0.003 & 0.096 & 0.054 \\
    4.155 & 2.622 & 0 & -0.218 & -4.848 & 0.003 & -0.002 \\
    -0.011 & -0.590 & 0 & 0.005 & 0.487 & 0 & 0 \\
    0 & 0 & -0.009 & 0 & 0 & -0.082 & -0.047 \\
    0.463 & -0.096 & -0.799 & -0.118 & -0.205 & -1.102 & 1.277 \\
    0.006 & -0.012 & -1.478 & 0.001 & -1.273 & 7.740 & -7.360
\end{smallmatrix}\right]
\]
with condition number $\sigma_1\sigma_n^{-1}\approx 130.40$
and Lyapunov matrix given by
\[
P = \left[\begin{smallmatrix}
    5.434 & 0.147 & -0.277 & -11.867 & 9.051 & -0.243 & -0.001 \\
    0.147 & 0.672 & -0.009 & -0.316 & 1.030 & 0.063 & 0.010 \\
    -0.277 & -0.009 & 0.077 & 0.998 & -0.513 & -0.104 & -0.018 \\
    -11.867 & -0.316 & 0.998 & 30.154 & -19.766 & -0.036 & -0.076 \\
    9.051 & 1.030 & -0.513 & -19.766 & 136.383 & -0.135 & 0.318 \\
    -0.243 & 0.063 & -0.104 & -0.036 & -0.135 & 0.933 & 0.133 \\
    -0.001 & 0.010 & -0.018 & -0.076 & 0.318 & 0.133 & 0.028
\end{smallmatrix}\right].
\]
The eigenvalues of the corresponding closed-loop subsystems are shown in Table~\ref{tab:ev_uav}. Fig.~\ref{fig:aero} depicts the Euclidean norm and the Lyapunov function along a closed-loop trajectory under a random switching signal, where the active mode is known at each time step.

\begin{figure}
    \centering
    \includegraphics[width=1.0\linewidth]{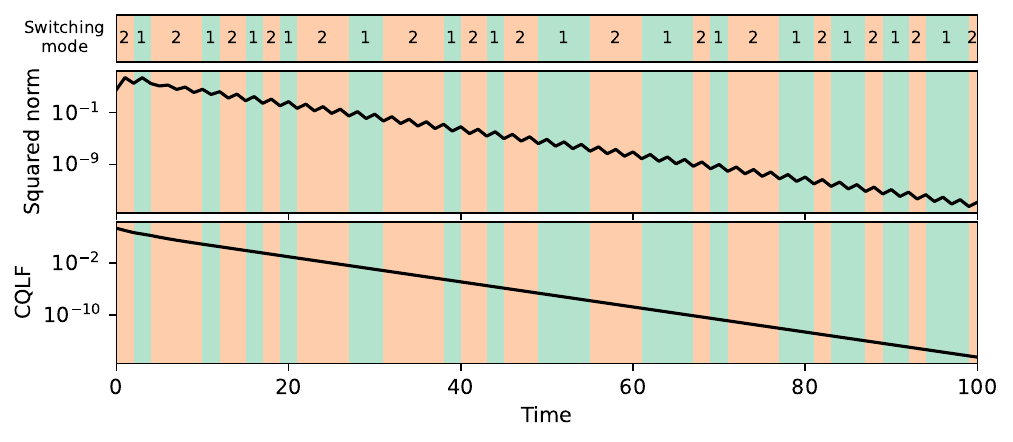}
    \caption{UAV. Euclidean norm and CQLF along a closed-loop state trajectory under a random switching signal.}
    \label{fig:aero}
\end{figure}
\begin{table} \caption{Eigenvalues of the closed-loop matrices $A_q+B_qF_q$ for the UAV switched system.} \label{tab:ev_uav} \centering {\setlength{\tabcolsep}{3pt} \begin{tabular}{@{}rrrrrrrr@{}} \toprule $q$ & $\lambda_1$ & $\lambda_2$ & $\lambda_3$ & $\lambda_4$ & $\lambda_5$ & $\lambda_6$ & $\lambda_7$ \\ \midrule $1$ & $-0.800$ & $-0.800$ & $-0.800$ & $-0.670$ & $0.602$ & $0.800$ & $0.800$ \\ $2$ & $-0.800$ & $0.004$ & $0.594$ & $0.596$ & $0.602$ & $0.611$ & $0.800$\\ \bottomrule \end{tabular}}  \end{table}

\section{Conclusion and outlook}
\label{sec:conclusion}
In this paper, we present a data-driven method for the stabilization of switched linear systems. Using input--state data collected during an offline phase and assuming that the collected data are informative, we derive state feedbacks that rectify the eigenstructure of switched linear systems. 
This allows us to obtain a common quadratic Lyapunov function purely based on the data.
Future work includes extending the approach to noisy data within the data-informativity framework and to noise during closed-loop operation, as well as incorporating data-driven mode detection. Furthermore, we plan to develop an efficient optimization algorithm to minimize the condition number of the  {eigenvector matrix} for large-scale problems.


\bibliographystyle{unsrt}
\bibliography{refs}

\end{document}